\documentclass[11pt]{article}
\usepackage{amsmath,amsthm}
\usepackage{amssymb}
\usepackage{enumerate}
\usepackage{graphicx,tikz}
\usepackage[mathscr]{eucal}
\usepackage[cm]{fullpage}
\usepackage{color}
\usepackage[english]{babel}
\usepackage[latin1]{inputenc}

\theoremstyle{plain}
\newtheorem{theorem}{Theorem}[section]
\newtheorem{proposition}[theorem]{Proposition}
\newtheorem{lemma}[theorem]{Lemma}

\newtheorem{notation}[theorem]{Notation}

\newtheorem{remark}{\textnormal{\textbf{Remark}}}

\theoremstyle{remark}

\def\im{\mathop{\mathrm{Im}}\nolimits}
\def\dom{\mathop{\mathrm{Dom}}\nolimits}

\def\rank{\mathop{\mathrm{rank}}\nolimits}
\def\id{\mathop{\mathrm{id}}\nolimits}
\numberwithin{equation}{section}
\def\N{\mathbb N}
\def\F{{\cal F}}
\def\C{{\cal C}}
\def\G{{\cal G}}

\title{The rank of the inverse semigroup of all partial automorphisms on a finite crown}

\author{Ilinka Dimitrova \\
\textit{Faculty of Mathematics and Natural Science}\\
\textit{South-West University "Neofit Rilski"}\\
\textit{2700 Blagoevgrad, Bulgaria}\\
\textit{email: ilinka\_dimitrova@swu.bg}\\
~~\\
J\"{o}rg Koppitz \\
\textit{Institute of Mathematics and Informatics}\\
\textit{Bulgarian Academy of Sciences}\\
\textit{1113 Sofia, Bulgaria}\\
\textit{e-mail: koppitz@math.bas.bg}}

\begin{document}

\maketitle

\begin{abstract}
For $n \in \N$, let $[n] = \{1, 2, \ldots, n\}$ be an $n$ - element set.
As usual, we denote by $I_n$ the symmetric inverse semigroup on $[n]$, i.e. the partial one-to-one transformation
semigroup on $[n]$ under composition of mappings. The crown (cycle) $\C_n$ is an $n$-ordered set with the partial order $\prec$ on $[n]$, where the only comparabilities are
$$1 \prec 2 \succ 3 \prec 4 \succ \cdots \prec n \succ 1 ~~\mbox{ or }~~ 1 \succ 2 \prec 3 \succ 4 \prec \cdots \succ n \prec 1.$$
We say that a transformation $\alpha \in I_n$ is order-preserving if $x \prec y$ implies that $x\alpha \prec y\alpha$, for all $x, y $ from the domain of $\alpha$.
In this paper, we study the inverse semigroup $IC_n$ of all partial automorphisms on a finite crown $\C_n$. We consider the elements, determine a generating set of minimal size and calculate the rank of $IC_n$.
\end{abstract}

\textit{Keywords:} finite transformation semigroup, inverse semigroup, crown, order-preserving injections,
generators, rank \\

2020 Mathematics Subject Classification: 20M20, 20M18
\\

\section{Introduction and Preliminaries}

~~\\
For $n \in \N$, let $[n] = \{1, 2, \ldots, n\}$ be an $n$ - element set.
As usual, we denote by $I_n$ the symmetric inverse semigroup on $[n]$, i.e. the partial one-to-one transformation
semigroup on $[n]$ under composition of mappings. The importance of $I_{n}$ to inverse semigroup theory may
be likened to that of the symmetric group $S_{n}$ to group theory. Every finite
inverse semigroup $S$ is embeddable in $I_{n}$, the analogue of Cayley's theorem
for finite groups, and to the regular representation of finite semigroups.
Thus, just as the study of symmetric, alternating and dihedral groups has
made a significant contribution to group theory, so has the study of various
subsemigroups of $I_{n}$, see for example \cite{DFKQ,DK1,Fern,GM,Umar}.

As usual, we denote by $<$ and $\leq$ the linear order on $\N$.
Let $n \in 2\N$. Consider the following two families of orders:

The fence (zig-zag) $\F_n$ is an $n$-ordered set with the partial order $\prec$ on $[n]$, where the only comparabilities are
$$1 \prec 2 \succ 3 \prec 4 \succ \cdots \prec n ~~\mbox{ or }~~ 1 \succ 2 \prec 3 \succ 4 \prec \cdots \succ n.$$
~~
The crown (cycle) $\C_n$ is an $n$-ordered set with the partial order $\prec$ on $[n]$, where the only comparabilities are
$$1 \prec 2 \succ 3 \prec 4 \succ \cdots \prec n \succ 1 ~~\mbox{ or }~~ 1 \succ 2 \prec 3 \succ 4 \prec \cdots \succ n \prec 1.$$

Every element of $\F_n$ and $\C_n$ is either maximal or minimal. Further, without any loss of generality, we assume $1$ to be a minimal element.
Clearly, the minimal elements are odd and maximal elements are even.\\

A fence $\F_n$ and a crown $\C_n$ are $n$-ordered sets with Hasse diagram isomorphic to the following:
\begin{center}
\includegraphics[scale=.35]{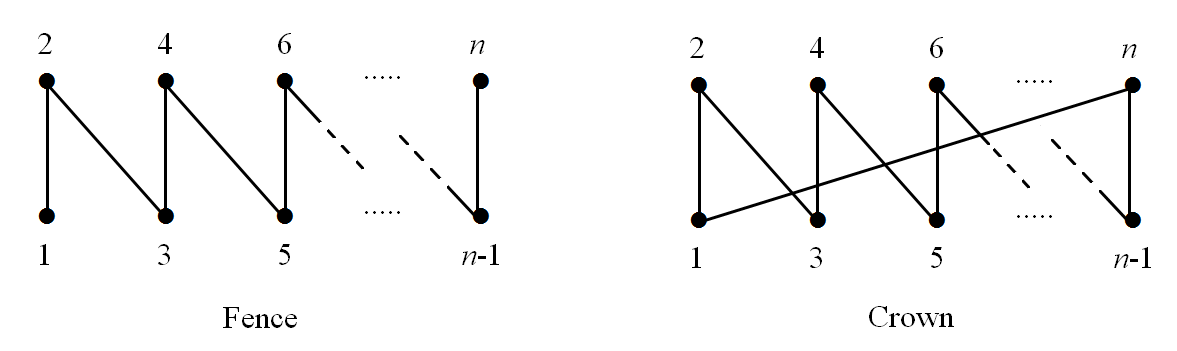}
\end{center}

For standard concepts in semigroup and symmetric inverse semigroup
theory, see for example \cite{How,Limp}. We denote by $\dom\alpha$ and $\im\alpha$ the domain and the image (range) of $\alpha \in I_{n}$, respectively. The natural number $\rank\alpha=\left\vert \im\alpha \right\vert$ is called the rank of $\alpha$. The inverse element of $\alpha$ is denoted by $\alpha^{-1}$. For a subset $X \subseteq [n]$, we denote by $\id|_X$ the identity mapping on $X$. Clearly, if $X = [n]$ then $\id|_{[n]} = \id$ is the identity mapping on $[n]$. For a subset $A\subseteq I_{n}$, we denote by $\left\langle A\right\rangle$ the subsemigroup of $I_{n}$ generated by $A$.

We say that a transformation $\alpha \in I_n$ is \textit{order-preserving} if $x \prec y$ implies that $x\alpha \prec y\alpha$, for all $x, y \in \dom \alpha$.
Order-preserving transformations of (finite) fences and crowns were first investigated by Currie and Visentin \cite{CV} in 1991. They calculated the number of order-preserving transformations of fences and crowns with an even number of elements, by using generating functions.
On the other hand, an exact formula for the order-preserving self-mappings of a fence, for any natural number $n$, was given by Rutkowski \cite{Rutk2} in 1992. Moreover, he calculated the number of strictly increasing mappings of fences in \cite{Rutk1}.
Later in 1995, Farley \cite{Farl} computed the number of order-preserving maps between fences and crowns.
Recently, several properties of monoids of order-preserving transformations of a fence were studied. In 2016, Chinram, Srithus and Tanyawong \cite{TSC} discussed the regular elements in these monoids. So-called coregular elements of these monoids were determined by Jendana and Srithus \cite{JS} in 2015.
In 2017, Dimitrova, Koppitz and Lohapan \cite{DKL} determined the relative rank of semigroups of partial transformations preserving a zig-zag order on $\mathbb{N}$ modulo a set containing all idempotents, all surjective transformations, all transformations with defect $\{1\}$ and a particular set of transformations with both infinite rank and infinite defect. In the same year, Lohapan and Koppitz \cite{KL} determined particular maximal regular subsemigroups of the semigroup of all partial transformations preserving a fence on $\mathbb{N}$ and showed that this semigroup has infinitely many maximal regular subsemigroups.
Fernandes, Koppitz and Musunthia \cite{FKM} calculated the rank of the semigroup of all order-preserving transformations on a finite fence in 2019.

We denote by $PFI_n$ (respectively, $PCI_n$) the subsemigroup of $I_n$ of all partial order-preserving injections of $\F_n$ (respectively, $\C_n$).
Note that the semigroups $PFI_n$ and $PCI_n$ are not inverse. For example
$$\alpha = \left(
             \begin{array}{cccc}
               1 & 2 & 5 & 6 \\
               3 & 2 & 5 & 4 \\
             \end{array}
           \right) \in PFI_6 ~\mbox{ but }~
\alpha^{-1} = \left(
             \begin{array}{cccc}
               2 & 3 & 4 & 5 \\
               2 & 1 & 6 & 5 \\
             \end{array}
           \right) \notin PFI_6
$$
and
$$\alpha = \left(
             \begin{array}{cccc}
               1 & 3 & 4 & 6  \\
               1 & 3 & 4 & 2 \\
             \end{array}
           \right) \in PCI_6 ~\mbox{ but }~
\alpha^{-1} = \left(
             \begin{array}{cccc}
               1 & 2 & 3 & 4 \\
               1 & 6 & 3 & 4 \\
             \end{array}
           \right) \notin PCI_6.$$

Let $IF_n$ (respectively, $IC_n$) be the set of all $\alpha \in PFI_n$ (respectively, $\alpha \in PCI_n$) such that $\alpha^{-1}\in PFI_n$
(respectively, $\alpha^{-1}\in PCI_n$).
Clearly, $IF_n$ (respectively, $IC_n$) is the set of all $\alpha \in PFI_n$ (respectively, $\alpha \in PCI_n$)
with $x \prec y$ if and only if $x\alpha \prec y\alpha$, for all $x, y \in \dom \alpha$.
Hence, $IF_n$ (respectively, $IC_n$) is the inverse semigroup of all partial automorphisms on a finite fence (respectively, crown).

In 2017, Dimitrova and Koppitz \cite{DK2} proved that the semigroup $IF_n$ is generated by its elements with rank $\geq n-2$. Moreover, they found the least generating set and calculated the rank of $IF_n$ for an even $n$. In 2021, Koppitz and Musunthia \cite{KM} solved the same problem for an odd $n$ and calculate the rank of $IF_n$.

In this paper, we will study the semigroup $IC_n$. We consider the elements, determine a generating set of minimal size and calculate the rank of $IC_n$.

\section{Generators and Rank}

In this section, we determine a generating set of minimal size and calculate the rank of $IC_n$.
First, we consider the elements in $IC_n$.
\begin{proposition}\label{pr1}
\rm Let $\alpha \in I_n$. Then $\alpha \in IC_n$ if and only if the following statements hold:

(1) if $x+1 \in \dom\alpha$ then $|x\alpha - (x+1)\alpha| = 1$ or $\{x\alpha, (x+1)\alpha\} = \{1,n\}$, for all $x \in \dom \alpha$;

(2) if $y+1 \in \im\alpha$ then $|y\alpha^{-1} - (y+1)\alpha^{-1}| = 1$ or $\{y\alpha^{-1}, (y+1)\alpha^{-1}\} = \{1,n\}$, for all $y \in \im \alpha$;

(3) if $x \in \dom\alpha$ such that $x$ and $x\alpha$ have different parity then $x+1 \notin \dom\alpha$.
\end{proposition}
\begin{proof}
Suppose that $\alpha \in IC_n$. Let $x \in \dom \alpha$ (respectively, $y \in \im\alpha$). Without loss of generality, we can assume that $x$ is odd (respectively, $y$ is odd). Then $x \prec x+1$ (respectively, $y \prec y+1$) implies $x\alpha \prec (x+1)\alpha$ (respectively, $y\alpha^{-1} \prec (y+1)\alpha^{-1}$), i.e. $(x+1)\alpha = x\alpha \pm 1$ or $\{x\alpha, (x+1)\alpha\} = \{1,n\}$ (respectively, $(y+1)\alpha^{-1} = y\alpha^{-1} \pm 1$ or $\{y\alpha^{-1}, (y+1)\alpha^{-1}\} = \{1,n\}$). This shows (1) and (2).
Let $x \in \dom\alpha$ such that $x$ and $x\alpha$ have different parity. Assume that $x+1 \in \dom\alpha$. Without loss of generality let $x$ be odd. Then $x \prec x+1$. Since $\alpha \in IC_n$, we obtain $x\alpha \prec (x+1)\alpha$. This implies $x\alpha$ is odd, a contradiction. This shows (3).

Suppose now that (1) - (3) hold. Without loss of generality, we can assume that $x$ is odd, i.e. $x \prec x+1$. Assume that $x+1 \in \dom \alpha$ and $(x+1)\alpha \prec x\alpha$, i.e. $x\alpha$ is even. Then $x$ and $x\alpha$ have different parity and thus, (3) implies $x+1 \notin \dom \alpha$, which is a contradiction.
Now, assume that $x+1 \in \dom \alpha$ and $x\alpha \nprec (x+1)\alpha$. Then $x\alpha \pm 1 \neq (x+1)\alpha$ and
$\{x\alpha, (x+1)\alpha\} \neq \{1,n\}$. This contradicts to (1). Therefore, we have $x+1 \notin \dom\alpha$ or $x\alpha \prec (x+1)\alpha$. This shows that from $x, x+1 \in \dom\alpha$ with $x \prec x+1$, it follows $x\alpha \prec (x+1)\alpha$.

Let $y \in \im\alpha$. Suppose that $y+1 \in \im\alpha$. Without loss of generality, we can assume that $y$ is odd. Then $y \prec y+1$. From (2), it follows that $y\alpha^{-1} \pm 1 = (y+1)\alpha^{-1}$ or $\{y\alpha^{-1}, (y+1)\alpha^{-1}\} = \{1,n\}$. Thus, we have $y\alpha^{-1} \prec (y+1)\alpha^{-1}$ or $(y+1)\alpha^{-1} \prec y\alpha^{-1}$. Assume that $(y+1)\alpha^{-1} \prec y\alpha^{-1}$. Then $y\alpha^{-1} \in \dom\alpha$ and $y\alpha^{-1}\alpha=y$ have different parity and by (3), we have $y\alpha^{-1}+1 \notin \dom\alpha$. Thus, $(y+1)\alpha^{-1} = y\alpha^{-1} - 1 \in \dom\alpha$, where $(y+1)\alpha^{-1}$ and $y+1$ have different parity, i.e. $(y+1)\alpha^{-1} +1 \notin \dom\alpha$ by (3). Hence, $y\alpha^{-1} = (y\alpha^{-1} - 1) + 1 = (y+1)\alpha^{-1} + 1 \notin \dom\alpha$, a contradiction. Therefore, from $y, y+1 \in \im\alpha$ with $y \prec y+1$, it follows $y\alpha^{-1} \prec (y+1)\alpha^{-1}$.
\end{proof}

It is easy to verify that
$$\left\{\id,
\left(
\begin{array}{c}
1 \\
2%
\end{array}%
\right),
\left(
\begin{array}{c}
2 \\
1%
\end{array}%
\right)
\right\}$$
is a generating set of minimal size of
$$IC_{2}=
\left\{\id,
\left(\begin{array}{c}
1 \\
1%
\end{array}%
\right),
\left(\begin{array}{c}
1 \\
2%
\end{array}%
\right),
\left(\begin{array}{c}
2 \\
1%
\end{array}%
\right),
\left(\begin{array}{c}
2 \\
2%
\end{array}%
\right),
\emptyset \right\},$$
where $\emptyset$ is the empty transformation.
This shows that
$$\rank IC_{2} = 3.$$

Next, let $n\geq 4$ and define
$$\sigma_1 = \left(
\begin{array}{ccccccc}
1 & 2 & 3 & \cdots & n-2 & n-1 & n \\
3 & 4 & 5 & \cdots &  n  &  1  & 2
\end{array}%
\right)$$
and
$$\sigma_2 = \left(
\begin{array}{ccccccc}
1 & 2 &  3  & \cdots & n-2 & n-1 & n \\
1 & n & n-1 & \cdots &  4  &  3  & 2
\end{array}%
\right).$$

Let $i, j \in [n]$. Then we put
$$i + j = \left\{\begin{array}{ll}
            i + j, & \mbox{if } 1 \leq i + j \leq n, \\
            i+j-n, & \mbox{if } i+j > n \\
          \end{array}\right.
~~\mbox{ and }~~
i - j = \left\{\begin{array}{ll}
            i - j, & \mbox{if } 1 \leq i - j \leq n, \\
            i-j+n, & \mbox{if } i-j < 1. \\
          \end{array}\right.$$

First, we will show that $\{\sigma_1, \sigma_2\}$ is a generating set for all transformations in $IC_{n}$ with rank $n$.
\begin{lemma}\label{leG1}
Let $\alpha \in IC_n$ with $\rank \alpha = n$. Then $\alpha \in \langle \sigma_1, \sigma_2 \rangle$.
\end{lemma}
\begin{proof}
If $\rank \alpha = n$ then $\dom\alpha = \im\alpha = [n]$ and by Proposition \ref{pr1}, we have $|x\alpha - (x+1)\alpha| = 1$ or $\{x\alpha, (x+1)\alpha\} = \{1,n\}$, for all $x \in [n]$. Therefore, $\alpha = \sigma_1^p$ or $\alpha = \sigma_1^p\sigma_2$, for some $p \in \{1,\ldots,\frac{n}{2}\}$.
\end{proof}

Now, let $\varepsilon_i = \id|_{[n]\setminus\{i\}}$ for $i \in [n]$, i.e.
$$\varepsilon_i = \left(
\begin{array}{cccccccc}
1 & 2 & \cdots & i-1 & i & i+1 & \cdots & n \\
1 & 2 & \cdots & i-1 & - & i+1 & \cdots & n
\end{array}%
\right).$$
Clearly, $\rank \varepsilon_i = n-1$. We will show that the set $\{\sigma_1, \sigma_2, \varepsilon_1, \varepsilon_n\}$ is a generating set for all transformations in $IC_{n}$ with rank $n-1$.

\begin{lemma}\label{leG2}
Let $\alpha \in IC_n$ with $\rank \alpha = n-1$. Then $\alpha \in \langle \sigma_1, \sigma_2, \varepsilon_1, \varepsilon_n \rangle$.
\end{lemma}
\begin{proof}
If $\rank \alpha = n-1$ then $\dom\alpha = [n] \setminus \{i\}$ and $\im\alpha = [n]\setminus \{k\}$ for some $i, k \in [n]$.

Let $i=k$. Then $\alpha = \varepsilon_i$ or
$$\alpha = \gamma_i = \left(
\begin{array}{cccccccc}
1 & 2 & \cdots & i-1 & i & i+1 & \cdots & n \\
2i-1 & 2i-2 & \cdots & i+1 & - & i-1 & \cdots & 2i
\end{array}%
\right).$$

If $\alpha = \varepsilon_i$ then
$$\alpha = \left\{\begin{array}{ll}
            \varepsilon_1, & \mbox{for } i=1, \\
            \varepsilon_n, & \mbox{for } i=n, \\
            \sigma_1^{\frac{n-i}{2}}\varepsilon_n\sigma_1^{\frac{i}{2}}, & \mbox{for } i \in 2\N \mbox{ and } 2 \leq i \leq n-2, \\
            \sigma_1^{\frac{n-i+1}{2}}\varepsilon_1\sigma_1^{\frac{i-1}{2}}, & \mbox{for } i \in 2\N+1 \mbox{ and } 3 \leq i \leq n-1.\\
          \end{array}\right.$$

If $\alpha = \gamma_i$ then
$$\alpha = \varepsilon_i\sigma_2\sigma_1^{i-1}.$$

Let $i \neq k$. Then
$$\alpha = \alpha_i^{(k)} = \left(
\begin{array}{cccccccc}
1 & \cdots & i-1 & i & i+1 & \cdots & n \\
k-i+1 & \cdots & k-1 & - & k+1 & \cdots & k-i
\end{array}%
\right)$$
or
$$\alpha = \gamma_{i}^{(k)}=\left(
\begin{array}{ccccccc}
1 & \cdots  & i-1 & i & i+1 & \cdots  & n \\
k+i-1 & \cdots  & k+1 & - & k-1 & \cdots  & k+i
\end{array}%
\right).$$
Notice that $i$ and $k$ have the same parity, i.e. both are even or both are odd.

If $\alpha = \alpha_i^{(k)}$ then
$$\alpha = \left\{\begin{array}{ll}
            \varepsilon_i\sigma_1^{\frac{k-i}{2}}, & \mbox{for } 1 \leq i < k \leq n, \\
            \varepsilon_i\sigma_1^{\frac{n+k-i}{2}}, & \mbox{for } 1 \leq k < i \leq n. \\
          \end{array}\right.$$

If $\alpha = \gamma_i^{(k)}$ then
$$\alpha = \alpha_{i}^{(n-k+2)}\sigma_{2}.$$
\end{proof}

Further, we define the following transformations with rank $n-2$:
$$\gamma_{i,j} = \left(
\begin{array}{cccccccccccc}
1 & 2 & \cdots & i-1 & i & i+1 & \cdots & j-1 & j & j+1 & \cdots & n \\
1 & 2 & \cdots & i-1 & - & j-1 & \cdots & i+1 & - & j+1 & \cdots & n
\end{array}%
\right),$$
for $1 \leq i < j \leq n$ and
$$\gamma_{i,j} = \left(
\begin{array}{ccccccccccc}
1 & \cdots & j-1 & j & j+1 & \cdots & i-1 & i & i+1 & \cdots & n \\
i+j-1 & \cdots & i+1 & - & j+1 & \cdots & i-1 & - & j-1 & \cdots & i+j
\end{array}%
\right),$$
for $1 \leq j < i \leq n$, whenever $i, j$ have the same parity, i.e. both are even or both are odd.
Let
$$\mathcal{A} = \{\sigma_1, \sigma_2, \varepsilon_1, \varepsilon_n\} \cup \{\gamma_{i,n} \mid i \in 2\N \mbox{ and } 4 \leq i \leq 2\left\lfloor \frac{n}{4}\right\rfloor\} \cup \{\gamma_{i,1} \mid i \in 2\N+1 \mbox{ and } 5 \leq i \leq 2\left\lfloor \frac{n}{4}\right\rfloor+1\}.$$

Let $\alpha \in IC_{n}$. Then we put
$$\chi(\alpha )=\{x\in \dom\alpha \mid x \mbox{ and } x\alpha \mbox{ have different parity} \}.$$

Let $IC_{n}^{p}$ be the set of all $\alpha \in IC_{n}$ with $\chi(\alpha)=\emptyset$ and let $IC_{n}^{\overline{p}}$ be the set of all $\alpha \in
IC_{n}$ with $\chi(\alpha)\neq \emptyset$.\\

First, we will show that $\mathcal{A}$ is a generating set of $IC_n^p$.
To accomplish this aim we start by proving a series of lemmas.

\begin{lemma}\label{leG3}
Let $1 \leq i < j \leq n$. Then $\gamma_{i,j} \in \langle \mathcal{A} \rangle$.
\end{lemma}
\begin{proof}
Let $i, j \in 2\N$. Then
$$\gamma_{2,n} = \varepsilon_2\varepsilon_n\sigma_2, ~~~ \gamma_{n-2,n} = \varepsilon_{n-2}\varepsilon_n, ~~~ \gamma_{i,n} \in \mathcal{A}, \mbox{ for } 4 \leq i \leq 2\left\lfloor \frac{n}{4}\right\rfloor,$$
$$\gamma_{i,n} = \sigma_1^{\frac{n-i}{2}}\gamma_{n-i,n}\sigma_1\sigma_2, \mbox{ for } 2\left\lfloor \frac{n}{4}\right\rfloor+2  \leq i \leq n-4, ~~~~
\gamma_{i,j} = \sigma_1^{\frac{n-j}{2}}\gamma_{n+i-j,n}\sigma_1^\frac{j}{2}, \mbox{ for } 2  \leq i < j \leq n-2.$$

Let $i, j \in 2\N-1$. Then
$$\gamma_{1,3} = \varepsilon_1\varepsilon_3, ~~~ \gamma_{1,n-1} = \varepsilon_1\varepsilon_{n-1}\sigma_2\sigma_1^{\frac{n}{2}-1},$$
$$\gamma_{1,j} = \sigma_2\gamma_{n-j+2,1}\sigma_2, \mbox{ for } 5 \leq j \leq n-3, ~~~~
\gamma_{i,j} = \sigma_1^{\frac{n-j+1}{2}}\gamma_{n-j+i+1,1}\sigma_1^{\frac{j-1}{2}}, \mbox{ for } 3 \leq i < j \leq n-1.$$
\end{proof}

\begin{lemma}\label{leG4}
Let $1 \leq j < i \leq n$. Then $\gamma_{i,j} \in \langle \mathcal{A} \rangle$.
\end{lemma}
\begin{proof}
Let $i, j \in 2\N$. Then
$$\gamma_{n,2} = \varepsilon_2\varepsilon_n, ~~~ \gamma_{n,n-2} = \varepsilon_{n-2}\varepsilon_n\sigma_2\sigma_1^{\frac{n}{2}-2},$$
$$\gamma_{n,j} = \sigma_1^{\frac{n-j}{2}}\gamma_{n-j,n}\sigma_1^{\frac{j}{2}}, \mbox{ for } 4  \leq j \leq n-4, ~~~~
\gamma_{i,j} = \sigma_1^{\frac{n-j}{2}}\gamma_{i-j,n}\sigma_1^{\frac{j}{2}}, \mbox{ for } 2  \leq j < i \leq n-2.$$

Let $i, j \in 2\N-1$. Then
$$\gamma_{3,1} = \varepsilon_1\varepsilon_3\sigma_2\sigma_1, ~~~ \gamma_{n-1,1} = \varepsilon_1\varepsilon_{n-1}, ~~~ \gamma_{i,1} \in \mathcal{A}, \mbox{ for } 5 \leq i \leq 2\left\lfloor \frac{n}{4}\right\rfloor+1,$$
$$\gamma_{i,1} = \sigma_1^{\frac{n-i+1}{2}}\gamma_{n-i+2,1}\sigma_2, \mbox{ for } 2\left\lfloor \frac{n}{4}\right\rfloor+3  \leq i \leq n-3, ~~~~
\gamma_{i,j} = \sigma_1^{\frac{n-j+1}{2}}\gamma_{i-j+1,1}\sigma_1^{\frac{j-1}{2}}, \mbox{ for } 3  \leq j < i \leq n-1.$$
\end{proof}

\begin{lemma}\label{leG5}
Let $i,j \in [n]$ have different parity and let
$$\gamma_{i,j}^- = \left(
\begin{array}{ccccccccccccc}
1 & 2 & \cdots & i-2 & i-1 & i & i+1 & \cdots & j-1 & j & j+1 & \cdots & n \\
1 & 2 & \cdots & i-2 &  -  & - & j-2 & \cdots &  i  & - & j+1 & \cdots & n
\end{array}%
\right),$$
for $1 \leq i < j \leq n$ and
$$\gamma_{i,j}^- = \left(
\begin{array}{ccccccccccccc}
  1   & \cdots & j-2 & j-1 & j & j+1 & \cdots & i-2 & i-1 & i & i+1 & \cdots &  n  \\
i+j-2 & \cdots & i+1 &  i  & - & j+1 & \cdots & i-2 &  -  & - & j-2 & \cdots & i+j-1
\end{array}%
\right),$$
for $1 \leq j < i \leq n$.
Then $\gamma_{i,j}^- \in \langle \mathcal{A} \rangle$.
\end{lemma}
\begin{proof}
The proof follows immediately from the equations $\gamma_{i,j}^- = \varepsilon_i\gamma_{i-1,j}$.
\end{proof}

\begin{lemma}\label{leG6}
Let $i,j \in [n]$ have different parity and let
$$\gamma_{i,j}^+ = \left(
\begin{array}{ccccccccccccc}
1 & 2 & \cdots & i-1 & i & i+1 & \cdots & j-1 & j & j+1 & j+2 & \cdots & n \\
1 & 2 & \cdots & i-1 & - & j & \cdots & i+2 & - & - & j+2 & \cdots & n
\end{array}%
\right),$$
for $1 \leq i < j \leq n$ and
$$\gamma_{i,j}^+ = \left(
\begin{array}{ccccccccccccc}
 1  & \cdots & j-2 & j-1 & j & j+1 & j+2 & \cdots & i-1 & i & i+1 & \cdots &  n  \\
i+j & \cdots & i+1 & i+2 & - &  -  & j+2 & \cdots & i-1 & - &  j  & \cdots & i+j+1
\end{array}%
\right),$$
for $1 \leq j < i \leq n$.
Then $\gamma_{i,j}^+ \in \langle \mathcal{A} \rangle$.
\end{lemma}
\begin{proof}
The proof follows immediately from the equations $\gamma_{i,j}^+ = \varepsilon_j\gamma_{i,j+1}$.
\end{proof}

\begin{lemma}\label{leG7}
Let $1 \leq i < j \leq n$ and let
$$\alpha_{i,j}^- = \left(
\begin{array}{ccccccccccccccc}
1 & 2 & \cdots & i-1 & i & i+1 & i+2 & i+3 & i+4 & \cdots & j-1 & j & j+1 & \cdots & n \\
1 & 2 & \cdots & i-1 & - &  -  &  -  & i+1 & i+2 & \cdots & j-3 & - & j+1 & \cdots & n
\end{array}%
\right),$$
$$\alpha_{i,j}^+ = \left(
\begin{array}{ccccccccccccccc}
1 & 2 & \cdots & i-1 & i & i+1 & i+2 & \cdots & j-1 & j & j+1 & j+2 & j+3 & \cdots & n \\
1 & 2 & \cdots & i-1 & - & i+3 & i+4 & \cdots & j+1 & - &  -  &  -  & j+3 & \cdots & n
\end{array}%
\right).$$
Then $\alpha_{i,j}^-, \alpha_{i,j}^+ \in \langle \mathcal{A} \rangle$.
\end{lemma}
\begin{proof}
We have
$$\alpha_{i,j}^- = \left\{\begin{array}{ll}
            \gamma_{i+1,j}\gamma_{i,j-1}, & \mbox{if } i \mbox{ and } j \mbox{ have different parity}, \\
            \varepsilon_{i+1}\gamma_{i,j}\gamma_{i,j-2}, & \mbox{if } i \mbox{ and } j \mbox{ have the same parity}
          \end{array}\right.$$
and
$$\alpha_{i,j}^+ = \left\{\begin{array}{ll}
            \gamma_{i,j+1}\gamma_{i+1,j+2}, & \mbox{if } i \mbox{ and } j \mbox{ have different parity}, \\
            \varepsilon_{j+1}\gamma_{i,j}\gamma_{i,j+2}, & \mbox{if } i \mbox{ and } j \mbox{ have the same parity}.
          \end{array}\right.$$
\end{proof}

\begin{lemma}\label{leG7'}
Let $1 \leq j < i \leq n$ and let
$$\alpha_{i,j}^- = \left(
\begin{array}{ccccccccccccccc}
 1  & 2 & \cdots & j-1 & j & j+1 & \cdots & i-1 & i & i+1 & i+2 & i+3 & \cdots & n-1 & n \\
n-1 & n & \cdots & j-3 & - & j+1 & \cdots & i-1 & - &  -  &  -  & i+1 & \cdots & n-3 & n-2
\end{array}%
\right),$$
$$\alpha_{i,j}^+ = \left(
\begin{array}{ccccccccccccccc}
1 & 2 & \cdots & j-1 & j & j+1 & j+2 & j+3 & \cdots & i-1 & i & i+1 & \cdots & n-1 & n \\
3 & 4 & \cdots & j+1 & - &  -  &  -  & j+3 & \cdots & i-1 & - & i+3 & \cdots &  1  & 2
\end{array}%
\right).$$
Then $\alpha_{i,j}^-, \alpha_{i,j}^+ \in \langle \mathcal{A} \rangle$.
\end{lemma}
\begin{proof}
The proof follows immediately from the equations $\alpha_{i,j}^- = \alpha_{j,i}^+\sigma_1^{\frac{n}{2}-1}$ and
$\alpha_{i,j}^+ = \alpha_{j,i}^-\sigma_1$.
\end{proof}

Now, we can prove $IC_n^p \subseteq \langle \mathcal{A} \rangle$.
For this, we need the concept of a maximal interval in a set $X \subseteq \N$. A set $I\subseteq X$ is called
maximal interval in $X$ if $I$ satisfies the
following properties:
\begin{itemize}
  \item [i)] $I$ is an interval of $X$ (i.e. $x<y\in I$ and $z\in \N$ with $x<z<y$ implies $z\in I$);
  \item [ii)] $x-1, x+1 \in I \cup (\N \setminus X)\cup \{0\}$ for all $x \in I$.
\end{itemize}

For sets $X, Y\subseteq \N$, we write $X<Y$ if $x<y$ for all $x\in X$ and all $y\in Y$.

\begin{notation}
Let $\alpha \in IC_{n}$. Then
$$I_{1}^{\alpha} < I_{2}^{\alpha} < \cdots < I_{k_{\alpha}}^{\alpha},$$
$k_{\alpha}\in [n]$, denotes the maximal intervals of $\dom\alpha$ and $\displaystyle\bigcup_{r=1}^{k_\alpha}
I_r^\alpha = \dom \alpha$.
\end{notation}

If $I$ is a maximal interval of $\dom \alpha$ for some partial transformation
$\alpha \in IC_{n}$ with $\{1,n\} \nsubseteq \dom\alpha$ and $\{1,n\} \nsubseteq \im\alpha$ then $I\alpha$ is a
maximal interval and $\alpha |_{I}$
is order-preserving or order-reversing (under the usual linear order $<$ on $\N$).

\begin{notation}
Let $\alpha \in IC_{n}$ with $\{1,n\} \nsubseteq \dom\alpha$ and $\{1,n\} \nsubseteq \im\alpha$.
We denote by $J_{r}^{\alpha} = I_{r}^{\alpha}\alpha$ for $r=1,\ldots,k_{\alpha}$ the maximal intervals of
$\im\alpha$. Obviously, $\displaystyle\bigcup_{r=1}^{k_\alpha} J_r^\alpha = \im \alpha$.
Then there is a permutation $\sigma_{\alpha}$ on $\{1,\ldots,k_{\alpha}\}$ such that
$$J_{1\sigma_{\alpha}}^{\alpha} < J_{2\sigma_{\alpha}}^{\alpha} < \cdots <
J_{k_{\alpha}\sigma_{\alpha}}^{\alpha}.$$
We denote by $t_r^\alpha$ (respectively, $q_r^\alpha$) the minimal (respectively, maximal) element in the set $J_r^\alpha$.
\end{notation}

\begin{lemma}\label{leG7''}
Let $\alpha \in IC_n^p$ such that $x+1 \in \dom\alpha$ for all $x \in [n]\setminus\dom\alpha$, $\{1,n\}\nsubseteq \dom\alpha$ and $\{1,n\}\nsubseteq \im\alpha$. Then $y+1 \in \im\alpha$ for all $y \in [n]\setminus\im\alpha$.
\end{lemma}
\begin{proof}
Suppose that $\dom\alpha = I_{1}^{\alpha}<I_{2}^{\alpha}<\cdots <I_{k_{\alpha}}^{\alpha}$. Then there are $x_{1},x_{2},\cdots,x_{k_{\alpha}}\in [n]\setminus \dom\alpha$ with
$I_{1}^{\alpha}<x_{1}<I_{2}^{\alpha}<x_{2}<\cdots <I_{k_{\alpha}}^{\alpha}<x_{k_{\alpha}}$ or $x_{1}<I_{1}^{\alpha}<x_{2}<I_{2}^{\alpha}<\cdots <x_{k_{\alpha}}<I_{k_{\alpha}}^{\alpha}$such that
$\dom\alpha \cup \{x_{1,}x_{2},\ldots,x_{k_{\alpha}}\}=[n]$. Note that $q_{1}^{\alpha}+1,\ldots,q_{k_{\alpha}}^{\alpha}+1\notin \im\alpha$ are pairwise different. Then we can calculate
that $n=k_{\alpha}+\left\vert \dom\alpha \right\vert = k_{\alpha}+\left\vert \im\alpha \right\vert$ and we have $[n]=\im\alpha \cup \{q_{r}^{\alpha}+1\mid r\in \{1,\ldots,k_{\alpha}\}\}$. Note that
$q_{r}^{\alpha }+2\neq q_{s}^{\alpha}+1$ for any $r,s\in \{1,\ldots,k_{\alpha}\}$. Hence, $q_{r}^{\alpha}+2\in \im\alpha$ for all $r\in \{1,\ldots,k_{\alpha}\}$, where
$\{q_{1}^{\alpha}+1,q_{2}^{\alpha}+1,\ldots,q_{k_{\alpha}}^{\alpha}+1\}=[n]\setminus \im\alpha$.
\end{proof}

\begin{proposition}\label{prG1}
\rm Let $\lambda \in IC_n^p$. Then $\lambda \in \langle \mathcal{A} \rangle$.
\end{proposition}
\begin{proof}
If $\rank \lambda \in \{n, n-1\}$ then $\lambda \in \langle \mathcal{A} \rangle$, by Lemmas \ref{leG1} and
\ref{leG2}.
Let $\rank \lambda \leq n-2$.

If $1\notin \dom\lambda$ or $n \notin \dom\lambda$ then we put $\varpi_{1}= \id$.
If $\{1,n\} \subseteq \dom\lambda$ with $I_{k_{\lambda}}^{\lambda}=\{a,\ldots,n\}$
then we put
$$\varpi_{1}=\left\{
\begin{array}{ll}
\sigma_{1}^{\frac{a-2}{2}}, & \mbox{if } a \mbox{ is even}, \\
\sigma_{1}^{\frac{a-1}{2}}, & \mbox{if } a \mbox{ is odd.}
\end{array}%
\right.$$
We observe that $(a-1)\notin \dom\lambda$ provides $1\notin \dom\varpi_{1}\lambda$ if $a$ is even
and $n\notin \dom\varpi_{1}\lambda$ if $a$ is odd. Hence, $\{1,n\}\nsubseteq \dom\varpi_{1}\lambda$.

If $1\notin \im\lambda$ or $n \notin \im\lambda$ then we put $\varpi_{2}=\id$.
If $\{1,n\} \subseteq \im\lambda$ with $\{b,\ldots,n\} \subseteq \im\lambda$ and $(b-1)\notin \im\lambda$
then we put
$$\varpi_{2}=\left\{
\begin{array}{ll}
\sigma_{1}^{\frac{n-b+2}{2}}, & \mbox{if } b \mbox{ is even}, \\
\sigma_{1}^{\frac{n-b+1}{2}}, & \mbox{if } b \mbox{ is odd.}
\end{array}%
\right.$$
Moreover, $(b-1)\notin \im\lambda$ implies $1\notin \im\lambda\varpi_{2}$ if $b$ is even and $n\notin
\im\lambda\varpi_{2}$ if $b$ is odd.
Hence, $\{1,n\}\nsubseteq \im\lambda\varpi_{2}$.

Since $\dom\varpi_1\lambda\varpi_2 \subseteq \dom\varpi_{1}\lambda$ and $\im\varpi_1\lambda\varpi_2 \subseteq
\im\lambda\varpi_{2}$, we have that
$\{1,n\}\nsubseteq \dom\varpi_1\lambda\varpi_2$ and $\{1,n\}\nsubseteq \im\varpi_1\lambda\varpi_2$.

We put now
\[\alpha =\varpi_{1}\lambda\varpi_{2}\]
and show that $\alpha \in \langle \mathcal{A} \rangle$.

Let $I = [n]\setminus \dom \alpha$. Then we put
$$\beta = \displaystyle\prod_{i\in I}\varepsilon_i.$$
Clearly, $\dom \beta = \dom \alpha$ and $\beta \in \langle \mathcal{A} \rangle$.

Let $s$ be the least integer $r \in \{1,\ldots,k_\alpha\}$ such that either $r\sigma_\alpha \neq r\sigma_\beta$ or $r\sigma_\alpha = r\sigma_\beta$ and $\alpha|_{I_{r\sigma_\alpha}^\alpha} \neq \beta|_{I_{r\sigma_\alpha}^\alpha}$.

Suppose that $s\sigma_\alpha \neq s\sigma_\beta$.

First, suppose that $s=1$.
Then we put
$$\beta = \left\{\begin{array}{ll}
                \beta\alpha_{q_{1\sigma_{\alpha}}^\beta+1}^{(q_{1\sigma_{\alpha}}^\alpha+1)}, & \mbox{if } \alpha|_{I_{1\sigma_\alpha}^\alpha} \mbox{ is order-preserving}, \\
                \beta\gamma_{q_{1\sigma_{\alpha}}^\beta+1}^{(t_{1\sigma_{\alpha}}^\alpha-1)}, & \mbox{if } \alpha|_{I_{1\sigma_\alpha}^\alpha} \mbox{ is order-reversing}
            \end{array}\right.$$
(i.e. we define a \textit{new} $\beta$; below we will made similar \textit{variables}'s redefinitions).
For the new $\beta$, we have $1\sigma_\alpha = 1\sigma_\beta$ and $\beta|_{I_{1\sigma_\alpha}^\alpha} = \alpha|_{I_{1\sigma_\alpha}^\alpha}$.

Now, suppose that $s > 1$.

If $t_{s\sigma_\beta}^\beta$ and $q_{s\sigma_\alpha}^\beta$ have the same parity then we put
$$\beta = \beta\gamma_{t_{s\sigma_\beta}^\beta-1,q_{s\sigma_\alpha}^\beta+1}.$$
For the last $\beta$, we have $s\sigma_\alpha = s\sigma_\beta$.

Further, suppose that $t_{s\sigma_\beta}^\beta$ and $q_{s\sigma_\alpha}^\beta$ have different parity.

If $t_{s\sigma_\beta}^\beta-2 \notin \im\beta$ or $q_{s\sigma_\alpha}^\beta+2 \notin \im\beta$ then we put
$$\beta = \left\{
\begin{array}{ll}
\beta\gamma_{t_{s\sigma_\beta}^\beta-1,q_{s\sigma_\alpha}^\beta+1}^+, & \mbox{if } q_{s\sigma_\alpha}^\beta+2 \notin \im\beta, \\
\beta\gamma_{t_{s\sigma_\beta}^\beta-1,q_{s\sigma_\alpha}^\beta+1}^-, & \mbox{if } t_{s\sigma_\beta}^\beta-2 \notin \im\beta. \\
\end{array}%
\right.$$
For the last $\beta$, we have $s\sigma_\alpha = s\sigma_\beta$.

Suppose that $t_{s\sigma_\beta}^\beta-2, q_{s\sigma_\alpha}^\beta+2 \in \im\beta$. Then we consider two cases.

\textit{Case 1.} Suppose that $x+1 \in \im\beta$ for all $x \in \{q_{(s-1)\sigma_{\beta}}^\beta+1, \ldots, n\}\setminus \im\beta$.\\
In this case, we have that $t_{s\sigma_\beta}^\beta-1 = q_{(s-1)\sigma_{\beta}}^\beta+1$. Moreover, $q_{(s-1)\sigma_{\beta}}^\beta = q_{(s-1)\sigma_{\alpha}}^\beta = q_{(s-1)\sigma_{\alpha}}^\alpha$ since $(s-1)\sigma_\alpha = (s-1)\sigma_\beta$ and $\alpha|_{I_{(s-1)\sigma_\alpha}^\alpha} = \beta|_{I_{(s-1)\sigma_\alpha}^\alpha}$.
Therefore, we have that $t_{s\sigma_\beta}^\beta$ and $q_{(s-1)\sigma_{\alpha}}^\beta$ have the same parity and thus, $q_{(s-1)\sigma_{\alpha}}^\beta$ and $q_{s\sigma_{\alpha}}^\beta$ have different parity. Because of Lemma \ref{leG7''}, we have that $q_{(s-1)\sigma_{\alpha}}^\beta$ and $t_{s\sigma_{\alpha}}^\beta$ have the same parity.

First, we show that $t_{s\sigma_{\alpha}}^\beta$ and $q_{s\sigma_{\alpha}}^\beta$ also have different parity.
Suppose the converse, i.e. $t_{s\sigma_{\alpha}}^\beta$ and $q_{s\sigma_{\alpha}}^\beta$ have the same parity. Then $\min I_{s\sigma_\alpha}^\beta$ and
$\max I_{s\sigma_\alpha}^\beta$ have the same parity. Since $\dom \beta = \dom \alpha$, we obtain that $\min I_{s\sigma_\alpha}^\alpha$ and
$\max I_{s\sigma_\alpha}^\alpha$ have the same parity. Thus, $t_{s\sigma_\alpha}^\alpha$ and $q_{s\sigma_\alpha}^\alpha$ also have the same parity, namely the parity of $q_{s\sigma_{\alpha}}^\beta$. Since $q_{(s-1)\sigma_{\alpha}}^\beta$ and $q_{s\sigma_{\alpha}}^\beta$ have different parity, we obtain that $q_{(s-1)\sigma_{\alpha}}^\beta$ has different parity with both $t_{s\sigma_{\alpha}}^\alpha$ and $q_{s\sigma_{\alpha}}^\alpha$.
But $q_{(s-1)\sigma_{\alpha}}^\beta$ and $q_{(s-1)\sigma_{\alpha}}^\beta+2 = q_{(s-1)\sigma_{\alpha}}^\alpha+2 = t_{s\sigma_{\alpha}}^\alpha$ have the same parity, a contradiction.
Therefore, we obtain that $t_{s\sigma_{\alpha}}^\beta$ and $q_{s\sigma_{\alpha}}^\beta$ have different parity. This implies that $t_{s\sigma_{\alpha}}^\beta$ and $q_{(s-1)\sigma_{\alpha}}^\beta$ have the same parity and thus, $t_{s\sigma_{\alpha}}^\beta$ and $q_{(s-1)\sigma_{\alpha}}^\beta+2 = t_{s\sigma_{\alpha}}^\alpha$ have the same parity. Moreover, we have that $q_{(s-1)\sigma_{\alpha}}^\beta$ and $q_{s\sigma_{\alpha}}^\alpha$ have different parity.

Clearly, $t_{s\sigma_{\alpha}}^\alpha \in J_{s\sigma_\beta}^\beta$ and $q_{s\sigma_\beta}^\beta < t_{s\sigma_{\alpha}}^\beta$.
Now, suppose that $q_{s\sigma_\beta}^\beta$ and $t_{s\sigma_{\alpha}}^\beta$ have different parity. Then $q_{s\sigma_\beta}^\beta$ and $q_{s\sigma_{\alpha}}^\beta$ have the same parity and we put
$$\beta = \beta\gamma_{q_{s\sigma_\beta}^\beta+1,q_{s\sigma_\alpha}^\beta+1}.$$
For the new $\beta$, we have that $q_{(s-1)\sigma_{\alpha}}^\alpha$ and $q_{s\sigma_\alpha}^\beta$ have the same parity and thus, $t_{s\sigma_{\beta}}^{\beta}$ and $q_{s\sigma_{\alpha}}^{\beta}$ have the same parity. Then we put
$$\beta = \beta\gamma_{t_{s\sigma_\beta}^\beta-1,q_{s\sigma_\alpha}^\beta+1}.$$
For the last $\beta$, we have $s\sigma_\alpha = s\sigma_\beta$.

Suppose that $q_{s\sigma_\beta}^\beta$ and $t_{s\sigma_{\alpha}}^\beta$ have the same parity. Since $t_{s\sigma_{\alpha}}^\beta$ and $q_{(s-1)\sigma_{\alpha}}^\beta$ have the same parity, we conclude that $q_{s\sigma_\beta}^\beta$ and $t_{s\sigma_\beta}^\beta$ have the same parity. Hence, $q_{s\sigma_\beta}^\alpha$ and $t_{s\sigma_\beta}^\alpha$ have the same parity and thus, $q_{s\sigma_\alpha}^\alpha$ and $t_{s\sigma_\beta}^\alpha$ have different parity, where $J_{s\sigma_\alpha}^\alpha < J_{s\sigma_\beta}^\alpha$.

Hence, there is $r \in \{(s+1)\sigma_\alpha, \ldots, k_\alpha\sigma_\alpha\}\setminus\{s\sigma_\beta\}$ such that $t_r^\alpha$ and $q_{s\sigma_{\alpha}}^{\alpha}$ have the same parity and $q_r^\alpha$ and $t_{s\sigma_{\beta}}^{\alpha}$ have the same parity. Thus, $t_r^\alpha$ and $q_r^\alpha$ as well as $t_r^\beta$ and $q_r^\beta$ have different parity.

Suppose that $J_{r}^\beta < J_{s\sigma_\alpha}^\beta$. If $t_r^\beta$ and $q_{s\sigma_\alpha}^\beta$ have the same parity then we put
$$\beta = \beta\gamma_{t_r^\beta-1,q_{s\sigma_\alpha}^\beta+1}.$$
If $t_r^\beta$ and $q_{s\sigma_\alpha}^\beta$ have different parity then $q_r^\beta$ and $t_{s\sigma_\alpha}^\beta$ have different parity and thus, there is $p \in \{(s+1)\sigma_\alpha, \ldots, k_\alpha\sigma_\alpha\}$ such that $J_{r}^\beta < J_{p}^\beta < J_{s\sigma_\alpha}^\beta$ where $t_p^\beta$ and $q_{s\sigma_\alpha}^\beta$ have the same parity by similar argumentations as above. Then we put
$$\beta = \beta\gamma_{t_p^\beta-1,q_{s\sigma_\alpha}^\beta+1}.$$

Suppose that $J_{r}^\beta > J_{s\sigma_\alpha}^\beta$. In order to obtain the new $\beta$, we have to make the dual construction as above for $q_r^\beta$ and $t_{s\sigma_\alpha}^\beta$.

In both cases for the last $\beta$, we have that $t_{s\sigma_{\beta}}^{\beta}$ and $q_{s\sigma_{\alpha}}^{\beta}$ have the same parity. Then we put
$$\beta = \beta\gamma_{t_{s\sigma_\beta}^\beta-1,q_{s\sigma_\alpha}^\beta+1}.$$
Finally, we obtain $s\sigma_\alpha = s\sigma_\beta$.\\

\textit{Case 2.} Suppose that there exists $x \in \{q_{s\sigma_{\beta}}^{\beta}+1,\ldots,n-1\}$ with $x, x+1 \notin \im\beta$.

If $x+2 \notin \im\beta$ then we put
$$\beta = \left\{
\begin{array}{ll}
\beta\alpha_{t_{s\sigma_\beta}^\beta-1,x}^+\gamma_{t_{s\sigma_\beta}^\beta+1,q_{s\sigma_\alpha}^\beta+1}^-, & \mbox{if } x < q_{s\sigma_\alpha}^\beta, \\
\beta\alpha_{t_{s\sigma_\beta}^\beta-1,x}^+\gamma_{t_{s\sigma_\beta}^\beta+1,q_{s\sigma_\alpha}^\beta+3}^-, & \mbox{if } x > q_{s\sigma_\alpha}^\beta.
\end{array}%
\right.$$
For the last $\beta$, we have $s\sigma_\alpha = s\sigma_\beta$.
In the following, we assume that $x+2 \in \im\beta$. \\
If $\vert J_{r\sigma_{\beta}}^{\beta}\vert$ is odd for all $r \in \{s,\ldots,k_{\alpha}\}$ then there is a least $x \in \{q_{s\sigma_{\alpha}}^{\beta}+1,\ldots,n-1\}$ with $x, x+1\notin \im\beta$ such that $x+1$ and $t_{s\sigma_{\beta}}^{\beta}-1$ have the same parity. Then we put
$$\beta = \beta\gamma_{t_{s\sigma_{\beta}}^{\beta}-1,x+1}.$$
For the new $\beta$, we have that $t_{s\sigma_{\beta}}^{\beta}$ and $q_{s\sigma_{\alpha}}^{\beta}$ have the same parity and we put
$$\beta = \beta\gamma_{t_{s\sigma_\beta}^\beta-1,q_{s\sigma_\alpha}^\beta+1}.$$
For the last $\beta$, we have $s\sigma_\alpha = s\sigma_\beta$.

Now, we suppose that there exists $r \in \{s,\ldots,k_{\alpha}\}$ such that $\vert J_{r\sigma_{\beta}}^{\beta}\vert$ is even.\\
If $\vert J_{s\sigma_{\alpha}}^{\beta}\vert $ is even and $x < J_{s\sigma_{\alpha}}^{\beta}$ then there is an $x' \in \{x, x+1\}$
such that $x'$ and $q_{s\sigma_{\alpha}}^{\beta}+1$ have the same parity and we put
$$\beta = \beta \gamma_{x',q_{s\sigma_{\alpha}}^{\beta}+1}.$$
If $\vert J_{s\sigma_{\alpha}}^{\beta}\vert$ is even and $x > J_{s\sigma_{\alpha}}^{\beta}$ then there is an $x' \in \{x, x+1\}$
such that $x'$ and $t_{s\sigma\alpha}^{\beta}-1$ have the same parity and we put
$$\beta = \beta \gamma_{t_{s\sigma\alpha}^{\beta}-1,x'}.$$
In both cases, for the new $\beta$, we have that $t_{s\sigma_{\beta}}^{\beta}$ and $q_{s\sigma_{\alpha}}^{\beta}$ have the same parity and we put
$$\beta = \beta\gamma_{t_{s\sigma_\beta}^\beta-1,q_{s\sigma_\alpha}^\beta+1}.$$
For the last $\beta$, we have $s\sigma_\alpha = s\sigma_\beta$.

Let $\vert J_{s\sigma_{\alpha}}^{\beta}\vert $ be odd. Suppose that $t_{s\sigma_{\beta}}^{\beta}$ and
$q_{r\sigma_{\beta}}^{\beta}$ have different parity.\\
If $x < J_{r\sigma_{\beta}}^{\beta}$ then there is an $x' \in \{x, x+1\}$ such that $x'$ and $q_{r\sigma_{\beta}}^{\beta}+1$ have the same parity and we put
$$\beta = \beta \gamma_{x',q_{r\sigma_{\beta}}^{\beta}+1}.$$
If $x > J_{r\sigma_{\beta}}^{\beta}$ then there is an $x' \in \{x, x+1\}$ such that $x'$ and $t_{r\sigma_\beta}^{\beta}-1$ have the same parity and we put
$$\beta = \beta \gamma_{t_{r\sigma_\beta}^{\beta}-1,x'}.$$
In both cases, for the new $\beta$, we obtain that $t_{s\sigma_{\beta}}^{\beta}$ and $q_{r\sigma_{\beta}}^{\beta}$ have the same parity.
Then we put
$$\beta = \beta \gamma_{t_{s\sigma_{\beta}}^{\beta}-1,q_{r\sigma_{\beta}}^{\beta}+1}.$$
For the last $\beta$, we have $\vert J_{s\sigma_{\beta}}^{\beta}\vert $ is even.

Suppose that $t_{s\sigma_{\beta}}^{\beta}-1$ and $x+1$ have different parity. Since $t_{s\sigma_{\alpha}}^{\beta}$ and $q_{s\sigma_{\alpha}}^{\beta}$ have the same parity (since $\vert J_{s\sigma_{\alpha}}^{\beta}\vert $ is odd), it is easy to check that there are $a, b\in [n]\setminus \im\beta$ with $q_{s\sigma_{\beta}}^{\beta}<a<x$ and $x+1<b$ such that $a$ and $b$ have the same parity. Then we put
$$\beta =\beta \gamma_{a,b}.$$

Finally, we can suppose that $t_{s\sigma_{\beta}}^{\beta}-1$ and $x+1$ have the same parity. Then we put
$$\beta = \beta \gamma_{t_{s\sigma_{\beta}}^{\beta}-1,x+1}.$$
For the new $\beta$, we have that $t_{s\sigma_{\beta}}^{\beta}$ and $q_{s\sigma_{\alpha}}^{\beta}$ have the same parity and we put
$$\beta = \beta\gamma_{t_{s\sigma_\beta}^\beta-1,q_{s\sigma_\alpha}^\beta+1}.$$
For the last $\beta$, we have $s\sigma_\alpha = s\sigma_\beta$.\\

Next, suppose that $\im \alpha|_{I_{s\sigma_\alpha}^\alpha} \neq \im \beta|_{I_{s\sigma_\alpha}^\alpha}.$
Then there are $a < b \in [n]$ and $m \in \{1,\ldots,n-1\}$ such that
$\im \alpha|_{I_{s\sigma_\alpha}^\alpha} = \{a,\ldots,b\}$ and either $J_{s\sigma_\alpha}^\beta = \im
\beta|_{I_{s\sigma_\alpha}^\alpha} = \{a-m,\ldots,b-m\}$ or
$J_{s\sigma_\alpha}^\beta = \{a+m,\ldots,b+m\}$.

First, suppose that $J_{s\sigma_\alpha}^\beta = \{a-m,\ldots,b-m\}$. Then there are $k \in \{0,1\}$ and $l \in \{0,1,\ldots,\frac{n}{2} - 1\}$ such that
$m = 2l +k$. Suppose that $l > 0$. Then there is $i \in [n]$ with $i > J_{s\sigma_\alpha}^\beta$ such that $i,i+1,i+2 \notin \im \beta$ or there are $i < j \in [n]$ with $i,j > J_{s\sigma_\alpha}^\beta$ such that $i,i+1,j,j+1 \notin \im \beta$.

If there is $i > J_{s\sigma_\alpha}^\beta$ such that $i,i+1,i+2 \notin \im \beta$ then we put
$$\beta = \beta\alpha_{a-m-1,i}^+.$$

If there are $i < j \in [n]$ of the same parity with $i,j > J_{s\sigma_\alpha}^\beta$ such that $i,i+1,j,j+1 \notin \im \beta$ then we put
$$\beta = \beta\gamma_{i+1,j+1}\alpha_{a-m-1,i}^+.$$
Then we have $i,i+1,i+2\notin \im\beta$ for the new $\beta$.

Suppose that there are $i<j\in [n]$ of different parity with $i,j>J_{s\sigma _{\alpha }}^{\beta }$ such that $i,i+1,j,j+1\notin \im\beta$ but $x+1\in \im\beta$
for all $x\in \{q_{s\sigma_{\alpha}}^{\beta},\ldots,n\}\setminus (\{i,j\}\cup \im\beta)$. Then there is $r\in \{s\sigma_{\beta},\ldots,k_{\alpha}\sigma_{\beta}\}$ such that
$\vert J_{r}^{\beta}\vert$ is even since $x+1\in \im\beta$ for all $x\in \{q_{s\sigma_{\alpha}}^{\alpha},\ldots,n\}\setminus \im\beta$ and there are
$r_{1},r_{2}\in \{s\sigma_{\beta},\ldots,k_{\alpha}\sigma_{\beta}\}$ such that $\{t_{r_{1}}^{\beta},q_{r_{1}}^{\beta},t_{r_{2}}^{\beta},q_{r_{2}}^{\beta}\}$ contains both even and odd integers.\\
If $J_{r}^{\beta}>j$ then either $i$ or $j$ have the same parity as $q_{r}^{\beta}+1$. Without loss of generality, let $j$ and $q_{r}^{\beta}+1$ have the same parity. Then we put
$$\beta =\beta \gamma_{j,q_r^\beta +1}.$$
If $J_{r}^{\beta}<i$ then either $i+1$ or $j+1$ have the same parity as $t_{r}^{\beta}-1$. Without loss of generality, let $i+1$ and $t_{r}^{\beta}-1$ have the same parity. Then we put
$$\beta =\beta \gamma_{t_{r}^{\beta}-1,i+1}.$$
In both cases, for the new $\beta$, we have $i$ and $j$ have the same parity.

Suppose now that $i<J_{r}^{\beta}<j$. Then there is $a\in \{t_{r}^{\beta}-1,q_{r}^{\beta}+1\}$ such that $a$ and $j+1$ have the same parity. Then we put
$$\beta =\beta \gamma_{a,j+1}$$
and we can conclude that $i$ and $j$ have the same parity in the new $\beta$.\\

After $l$ such steps, we obtain $l=0$, i.e. $m=k$ and thus, $J_{s\sigma_\alpha}^\beta = \{a-k,\ldots,b-k\}$. If $k=0$ then $\im \beta|_{I_{s\sigma_\alpha}^\alpha} = \im \alpha|_{I_{s\sigma_\alpha}^\alpha}.$ If $k=1$ then $a$ and $b$ have different parity and $b,b+1 \notin \im\beta$ or there is $j > b$ with $j,j+1 \notin \im \beta$.

If $b,b+1 \notin \im\beta$ then we put
$$\beta = \beta\gamma_{a-2,b+1}.$$
Then $\im\beta|_{I_{s\sigma_{\alpha}}^{\alpha}}=\im\alpha|_{I_{s\sigma_{\alpha}}^{\alpha}}$ for the new $\beta$.

If there is $j > b$ with $j,j+1 \notin \im \beta$ such that $b$ and $j$ have different parity then we put
$$\beta = \beta\gamma_{b,j+1}\gamma_{a-2,b+1}.$$
We have $b,b+1\notin \im\beta$ for the new $\beta$.

If there is $j > b$ with $j,j+1 \notin \im \beta$ such that $b$ and $j$ have the same parity.
Moreover, $x$ and $b$ have the same parity for all $x\in \{b,\ldots,n\}$ with $x,x+1\notin \im\beta$. Then there is $r\in \{(s+1)\sigma_{\beta},\ldots,k_{\alpha}\sigma_{\beta}\}$ such that
$\vert J_{r}^{\beta}\vert$ is even and $J_{r}^{\beta}<j$. Then there is $a\in \{t_{r}^{\beta}-1,q_{r}^{\beta}+1\}$ such that $a$ and $j+1$ have the same parity. Then we put
$$\beta =\beta \gamma_{a,j+1}.$$
We can calculate that $b$ and $j$ have different parity in the new $\beta$.

Let now $J_{s\sigma_\alpha}^\beta = \{a+m,\ldots,b+m\}$. Then there are $k \in \{0,1\}$ and $l \in \{0,1,\ldots,\frac{n}{2} - 1\}$ such that
$m = 2l +k$. Suppose that $l > 0$. It is easy to verify that $a+m-1, a+m-2, a+m-3 \notin \im\beta$. In this case, we put
$$\beta = \beta\alpha_{a+m-3,b+m+1}^-.$$

After $l$ such steps, we obtain $l=0$, i.e. $m=k$ and thus, $J_{s\sigma_\alpha}^\beta = \{a+k,\ldots,b+k\}$. If $k=0$ then $\im \beta|_{I_{s\sigma_\alpha}^\alpha} = \im \alpha|_{I_{s\sigma_\alpha}^\alpha}.$ If $k=1$ then $a$ and $b$ have different parity and $a, a-1 \notin \im\beta$. In this case, we put
$$\beta = \beta\gamma_{a-1,b+2}.$$
It is easy to verify that $\im \alpha|_{I_{s\sigma_\alpha}^\alpha} = \im \beta|_{I_{s\sigma_\alpha}^\alpha}$ for the new $\beta$.

Finally, if $\alpha|_{I_{s\sigma_\alpha}^\alpha} \neq \beta|_{I_{s\sigma_\alpha}^\alpha}$ then $a\alpha^{-1}\beta = b$ and $b\alpha^{-1}\beta = a$, i.e. $a$ and $b$ have the same parity. In this case, we put
$$\beta = \beta\gamma_{a-1,b+1}.$$

We repeat the procedure until $r\sigma_\alpha = r\sigma_\beta$ and $\alpha|_{I_{r\sigma_\alpha}^\alpha} = \beta|_{I_{r\sigma_\alpha}^\alpha}$
for all $r \in \{1,\ldots,k_\alpha\}$, i.e. $\beta=\alpha$.\\

Therefore, by Lemmas \ref{leG3} -- \ref{leG7'}, we may deduce that $\alpha \in \langle \mathcal{A} \rangle$.
Since $\varpi_{1}^{-1}$, $\varpi_{2}^{-1}\in \langle \mathcal{A} \rangle $ and because both $\varpi_{1}$
and $\varpi_{2}$ are bijections, we can conclude that
$$\lambda = (\varpi_{1}^{-1}\varpi_{1})\lambda(\varpi_{2}\varpi_{2}^{-1}) =
\varpi_{1}^{-1}(\varpi_{1}\lambda\varpi_{2})\varpi_{2}^{-1} =
\varpi_{1}^{-1}\alpha\varpi_{2}^{-1} \in \langle \mathcal{A} \rangle.$$
\end{proof}

Next, we consider the set $IC_{n}^{\overline{p}}$.
\begin{lemma}\label{leG8}
Let $\alpha \in IC_{n}^{\overline{p}}$ and let $x\in \chi(\alpha)$. Then $x-1,x+1\notin \dom\alpha$ and $x\alpha-1, x\alpha+1 \notin \im\alpha$.
\end{lemma}
\begin{proof}
From Proposition \ref{pr1} (3), we have $x+1\notin \dom\alpha$. Assume that $x-1\in \dom\alpha $. Then $x\prec x-1$ if and only if $x\alpha
\prec (x-1)\alpha $ as well as $x-1\prec x$ if and only if $(x-1)\alpha
\prec x\alpha $. This implies that $x$ and $x\alpha $ have the same parity,
a contradiction. Since $\alpha^{-1}\in IC_{n}$, we get $x\alpha-1,x\alpha+1 \notin \im\alpha$ by similar arguments.
\end{proof}

Let define the following transformations of $IC_{n}^{\overline{p}}$ with rank $n-3$:
$$\delta_{1}^{o} = \left(
\begin{array}{ccccccc}
1 & 2 & 3 & 4 & \cdots  & n-1 & n \\
- & 1 & - & 4 & \cdots  & n-1 & -
\end{array}%
\right),$$

$$\delta_{i}^{o}=\left(
\begin{array}{cccccccccc}
1 & 2 & 3 &   4  & \cdots & n-2i+1 & n-2i+2 & n-2i+3 & \cdots & n \\
- & 1 & - & 2i+2 & \cdots &  n-1   &    -   &    3   & \cdots & 2i
\end{array}%
\right),$$
for $2 \leq i \leq \frac{n}{2}-1$,

$$\delta_{1}^{e}=\left(
\begin{array}{ccccccc}
1 & 2 & 3 & 4 & \cdots  & n-1 & n \\
2 & - & - & 4 & \cdots  & n-1 & -
\end{array}%
\right),$$

$$\delta_{i}^{e}=\left(
\begin{array}{cccccccccc}
1 & 2 & 3 & \cdots  & 2i & 2i+1 & 2i+2 & \cdots  & n-1 & n \\
2 & - & n-2i+3 & \cdots  & n & - & 4 & \cdots  & n-2i+1 & -
\end{array}%
\right),$$
for $2 \leq i \leq \frac{n}{2}-1$,
as well as two transformations with rank $\frac{n}{2}$:
$$\eta_{1} = \left(
\begin{array}{cccccccc}
1 & 2 & 3 & 4 & 5 & \cdots  & n-1 & n \\
2 & - & 4 & - & 6 & \cdots  & n & -%
\end{array}%
\right),$$

$$\eta_{2} = \left(
\begin{array}{cccccccc}
1 & 2 & 3 & 4 & 5 & \cdots  & n-1 & n \\
- & 1 & - & 3 & - & \cdots  & - & n-1%
\end{array}%
\right).$$

Denote by $2[n]$ (respectively, $2[n]-1$) the set of all even (respectively, odd) numbers of the set $[n]$, i.e.
$$2[n] = \{x \in [n] \mid x \mbox{ is even}\}, ~~~~ 2[n]-1 = \{x \in [n] \mid x \mbox{ is odd}\}.$$
Obviously, $|2[n]| = |2[n]-1| = \frac{n}{2}$.\\

Let $\mathcal{B} = \{\delta_{1}^{o}, \delta_{1}^{e}, \eta_1, \eta_2\} \cup \{\delta_{i}^{o}, \delta_{i}^{e} \mid 2 \leq i \leq \lfloor \frac{n}{4}\rfloor\}$.

\begin{lemma}\label{leG9}
Let $i\in \{\left\lfloor \frac{n}{4}\right\rfloor +1,\ldots,\frac{n}{2}-1\}$. Then $\delta_{i}^{o}, \delta_{i}^{e} \in
\left\langle \mathcal{B}, IC_{n}^{p}\right\rangle$ and $(\delta_{i}^{o})^{-1}=\delta_{i}^{e}$ and $(\delta_{i}^{e})^{-1}=\delta_{i}^{o}$.
\end{lemma}
\begin{proof}
We can calculate that $1 \leq \frac{n}{2}-i \leq \left\lfloor \frac{n}{4}\right\rfloor$ and $\delta _{i}^{o}=\gamma _{3,1}\delta _{\frac{n}{2}-i}^{o}\gamma _{2,n}$ as well as $\delta_{i}^{e}=\gamma_{2,n}\delta_{\frac{n}{2}-i}^{e}\gamma_{3,1}$. Moreover, we have
$$\delta_{i}^{e}=\gamma_{2,n}\delta_{\frac{n}{2}-i}^{e}\gamma_{3,1}=(\gamma_{2,n})^{-1}(\delta_{\frac{n}{2}-i}^{o})^{-1}(\gamma_{3,1})^{-1}=
(\gamma_{3,1}\delta_{\frac{n}{2}-i}^{o}\gamma_{2,n})^{-1}=(\delta_{i}^{o})^{-1}.$$
In particular, this implies $\delta_{i}^{o}=(\delta_{i}^{e})^{-1}$.
\end{proof}

\begin{lemma}\label{leG10}
Let $\alpha \in IC_{n}^{\overline{p}}$ with $\chi(\alpha )\subset \dom\alpha$. Then there are $i,j\in [n]$ of the same parity such that $i\in \chi(\alpha)$
and $j\notin \dom\alpha$.
\end{lemma}
\begin{proof}
Since $\alpha \in IC_{n}^{\overline{p}}$, there is $i\in [n]$ such that $i\in \chi(\alpha)$.
Let $i$ be even and let $C = \{x \in 2[n] \cap \dom\alpha \mid x\alpha \in 2[n]\}$. Assume that $2[n] \subseteq \dom \alpha$.
Since $\chi(\alpha )\subset \dom\alpha$, we have $C \neq \emptyset$ and $|C\alpha| = |C| = \frac{n}{2} - |2[n]\setminus C|$.
Because of Lemma \ref{leG8}, we have $x\alpha -1,x\alpha +1\notin \im\alpha $
for all $x\in 2[n]\setminus C$, where both $x\alpha -1$ and $x\alpha +1$ have the same parity as the
elements in $C\alpha$, i.e. they are even.
Since $C\alpha \neq \emptyset$, it is easy to verify that $\left\vert
\{x\alpha -1,x\alpha +1\mid x\in 2[n]\setminus C\} \right\vert \geq \left\vert 2[n]\setminus C \right\vert +1$.
This provides $\left\vert C\alpha \right\vert \leq \frac{n}{2}-(\left\vert 2[n]\setminus C \right\vert +1)$, a contradiction.
Hence, $2[n] \nsubseteq \dom \alpha$, i.e. there is $j\in 2[n]\setminus \dom\alpha$.
Analogously, if $i\in \chi(\alpha)$ is odd one can show that $2[n]-1 \nsubseteq \dom \alpha$, i.e. there is $j\in (2[n]-1)\setminus \dom\alpha$.
\end{proof}

\begin{proposition}\label{prG2}
Let $\alpha \in IC_{n}^{\overline{p}}$. Then $\alpha \in \left\langle
\mathcal{B}, IC_{n}^{p}\right\rangle $.
\end{proposition}
\begin{proof}
Since $\alpha \in IC_{n}^{\overline{p}}$, we have $\chi(\alpha) \neq \emptyset$.
If $\chi(\alpha) = \dom\alpha$ then $\alpha  \in \langle\eta_1, \eta_2, IC_n^p\rangle$ since $\alpha = \eta_i\omega$ for some $i \in \{1,2\}$ and a suitable
permutation $\omega \in IC_n^p$ on $2[n]$ and $2[n]-1$, respectively.

Let $\chi(\alpha) \subset \dom\alpha$.
Then there are $i\in \chi(\alpha)$ and $j\in [n]\setminus \dom\alpha $ such that $i$ and $j$ have the same parity
by Lemma \ref{leG10}. We put $k=\frac{j-i}{2}$. Recall, that
$$j-i = \left\{\begin{array}{ll}
            j-i, & \mbox{if }  j-i \geq 1, \\
            j-i+n, & \mbox{if } j-i < 1. \\
          \end{array}\right.$$

Suppose that $i$ is odd. Then $1\sigma _{1}^{\frac{i-1}{2}}=i$. Therefore, $\dom(\sigma_{1}^{\frac{i-1}{2}}\alpha )=\{x-i+1 \mid x\in \dom\alpha \}$. In
particular, $n, 2, j-i+1 \notin \dom(\sigma_{1}^{\frac{i-1}{2}}\alpha)$ since $i-1, i+1, j \notin \dom \alpha$.
It is easy to verify that $j-i+1$ is odd, i.e. $2k+1=j-i+1$ for some $k\in \{1,\ldots, \frac{n}{2}-1\}$. Moreover, we have $[n]\setminus\{n, 2, 2k+1\} = \im\delta_{k}^{o}$.
Hence, $\dom(\sigma_{1}^{\frac{i-1}{2}}\alpha)\subseteq \im\delta_{k}^{o}$.
Therefore, we can calculate that
$$(\delta_{k}^{o})^{-1}\delta_{k}^{o}\sigma_{1}^{\frac{i-1}{2}}\alpha = (\id|_{\im\delta_{k}^{o}})\sigma_{1}^{\frac{i-1}{2}}\alpha = \sigma_{1}^{\frac{i-1}{2}}\alpha$$
and thus,
$$(\delta_{k}^{o}\sigma_{1}^{\frac{i-1}{2}})^{-1}(\delta_{k}^{o}\sigma_{1}^{\frac{i-1}{2}})\alpha =
(\sigma_{1}^{\frac{i-1}{2}})^{-1}(\delta_{k}^{o})^{-1}\delta_{k}^{o}\sigma_{1}^{\frac{i-1}{2}}\alpha =
(\sigma_{1}^{\frac{i-1}{2}})^{-1}\sigma_{1}^{\frac{i-1}{2}}\alpha = \id\alpha = \alpha.$$
From Lemma \ref{leG9}, we have $(\delta_{k}^{o})^{-1}=\delta_{k}^{e}\in \left\langle \mathcal{B}, IC_{n}^{p}\right\rangle$.
Moreover, we can observe that $(\sigma_{1}^{\frac{i-1}{2}})^{-1} \in IC_{n}^{p}$.
Thus, $(\delta_{k}^{o}\sigma_{1}^{\frac{i-1}{2}})^{-1}=(\sigma_{1}^{\frac{i-1}{2}})^{-1}(\delta_{k}^{o})^{-1}\in \left\langle\mathcal{B}, IC_{n}^{p} \right\rangle $.
On the other hand, $2(\delta_{k}^{o}\sigma_{1}^{\frac{i-1}{2}}\alpha) = 1(\sigma_{1}^{\frac{i-1}{2}}\alpha) = i\alpha$ is even and $|\chi(\delta_{k}^{o}\sigma_{1}^{\frac{i-1}{2}})| = 1$.
Thus, $|\chi(\delta_{k}^{o}\sigma_{1}^{\frac{i-1}{2}}\alpha)| = |\chi(\alpha)| - 1$. Note that $\rank\alpha =\rank(\delta_{k}^{o}\sigma_{1}^{\frac{i-1}{2}}\alpha)$.

We put $\beta = \delta_{k}^{o}\sigma_{1}^{\frac{i-1}{2}}\alpha$ and repeat the same procedure until we obtain $\beta \in IC_{n}^p$ or $\beta \in IC_{n}^{\overline{p}}$
with $\chi(\beta)\subseteq 2[n]$, where $\rank\beta =\rank\alpha$.

If $\beta \in IC_{n}^{\overline{p}}$ then there are $i\in \chi(\alpha)$ and $j\in [n]\setminus \dom\alpha $ such that $i$ and $j$ are even.
Then $2\sigma_{1}^{\frac{i-2}{2}}=i$. Therefore, $\dom(\sigma_{1}^{\frac{i-2}{2}}\beta)=\{x-i+2 \mid x\in \dom\beta\}$.
In particular, $1,3,j-i+2\notin \dom(\sigma_{1}^{\frac{i-2}{2}}\beta)$ since $i-1, i+1, j \notin \dom \alpha$.
It is easy to verify that $n-2(\frac{n}{2}-k)+2=j-i+2$ for some $k\in \{1,\ldots,\frac{n}{2}\}$. Moreover, we have $[n]\setminus\{1, 3, n-2(\frac{n}{2}-k)+2\} = \im\delta_{\frac{n}{2}-k}^{e}$.
Hence, $\dom(\sigma_{1}^{\frac{i-2}{2}}\beta)\subseteq \im\delta_{\frac{n}{2}-k}^{e}$.
Therefore, we can calculate that
$$(\delta_{\frac{n}{2}-k}^{e})^{-1}\delta_{\frac{n}{2}-k}^{e}\sigma_{1}^{\frac{i-2}{2}}\beta =
\id|_{\im\delta_{\frac{n}{2}-k}^{e}}\sigma_{1}^{\frac{i-2}{2}}\beta = \sigma_{1}^{\frac{i-2}{2}}\beta.$$
and thus,
$$(\delta_{\frac{n}{2}-k}^{e}\sigma_{1}^{\frac{i-2}{2}})^{-1}(\delta_{\frac{n}{2}-k}^{e}\sigma_{1}^{\frac{i-2}{2}})\beta =
(\sigma_{1}^{\frac{i-2}{2}})^{-1}(\delta_{\frac{n}{2}-k}^{e})^{-1}\delta_{\frac{n}{2}-k}^{e}\sigma_{1}^{\frac{i-2}{2}}\beta =
(\sigma_{1}^{\frac{i-2}{2}})^{-1}\sigma_{1}^{\frac{i-2}{2}}\beta = \id\beta = \beta.$$
From Lemma \ref{leG9}, we have $(\delta_{\frac{n}{2}-k}^{e})^{-1}=\delta_{\frac{n}{2}-k}^{o} \in \left\langle \mathcal{B}, IC_{n}^{p}\right\rangle$.
Moreover, we can observe that $(\sigma_{1}^{\frac{i-2}{2}})^{-1} \in IC_{n}^{p}$.
Thus, $(\delta_{\frac{n}{2}-k}^{e}\sigma_{1}^{\frac{i-2}{2}})^{-1}=(\sigma_{1}^{\frac{i-2}{2}})^{-1}(\delta_{\frac{n}{2}-k}^{e})^{-1}\in \left\langle \mathcal{B}, IC_{n}^{p} \right\rangle $.
On the other hand, $1(\delta_{\frac{n}{2}-k}^{e}\sigma_{1}^{\frac{i-2}{2}}\beta) = 2(\sigma_{1}^{\frac{i-2}{2}}\beta) = i\beta$ is odd and $|\chi(\delta_{\frac{n}{2}-k}^{e}\sigma_{1}^{\frac{i-2}{2}})| = 1$.
Thus, $|\chi(\delta_{\frac{n}{2}-k}^{e}\sigma_{1}^{\frac{i-2}{2}}\beta)| = |\chi(\beta)|-1$.

We put $\gamma = \delta_{\frac{n}{2}-k}^{e}\sigma_{1}^{\frac{i-2}{2}}\beta$ and repeat the procedure until we obtain $\gamma \in IC_{n}^p$.\\

This shows that there are $l, m \in \{1,\ldots,\frac{n}{2}-1\}$ and $\pi_{1},\ldots,\pi_{l}, \tau_1,\ldots,\tau_m \in \left\langle \mathcal{B}, IC_{n}^{p}\right\rangle$
such that
$$\tau_m\cdots\tau_1\pi_{l}\cdots \pi_{1}\alpha \in IC_{n}^{p}.$$
Notice that for $s = 1,\ldots,l$ and $t = 1,\ldots,m$, there are $i_{s},i_{t},k_{s},k_{t}\in [n]$ such that $\pi_{s}=\delta_{k_{s}}^{o}\sigma_{1}^{\frac{i_{s}-1}{2}}$ and
$\tau_{t}=\delta_{\frac{n}{2}-k_{t}}^{e}\sigma_{1}^{\frac{i_{t}-2}{2}}$.
Moreover, we have $\pi_{1}^{-1},\ldots,\pi_{l}^{-1},\tau_{1}^{-1},\ldots,\tau_{m}^{-1} \in \left\langle \mathcal{B}, IC_{n}^{p}\right\rangle$ and \\
$\rank(\pi_{1}^{-1}\ldots \pi_{l}^{-1}\tau_{1}^{-1}\ldots \tau_{m}^{-1}\tau_{m}\ldots \tau_{1}\pi_{l}\ldots \pi_{1}\alpha)=\rank\alpha$.
Therefore, we obtain
$$\alpha = \pi_{1}^{-1}\cdots\pi_{l}^{-1}\tau_1^{-1}\cdots\tau_m^{-1}(\tau_m\cdots\tau_1\pi_{l}\cdots\pi_{1}\alpha) \in \left\langle \mathcal{B}, IC_{n}^{p}\right\rangle.$$
\end{proof}

Let $\G = \mathcal{A} \cup \mathcal{B}$. From Propositions \ref{prG1} and \ref{prG2}, we have
\begin{theorem}\label{thG1}
  $IC_n = \langle \G \rangle$.
\end{theorem}

Next, we will show that $\G$ is a generating set of $IC_n$ of minimal size.
\begin{proposition}\label{prG4}
$\rank IC_{n}\leq 4(\left\lfloor \frac{n}{4}\right\rfloor +1) = |\G|$.
\end{proposition}
\begin{proof}
We have $\left\vert \mathcal{A} \right\vert = 2+2+\left\lfloor \frac{n}{4}\right\rfloor-1+\left\lfloor \frac{n}{4}\right\rfloor -1=2+2\left\lfloor
\frac{n}{4}\right\rfloor$, $\left\vert \mathcal{B}\right\vert = 4+2\left\lfloor \frac{n}{4}\right\rfloor -2=2\left\lfloor \frac{n}{4}\right\rfloor +2$ and
$\left\vert \mathcal{G}\right\vert =\left\vert \mathcal{A}\right\vert +\left\vert \mathcal{B}\right\vert =2(2+2\left\lfloor
\frac{n}{4}\right\rfloor )=4(\left\lfloor \frac{n}{4}\right\rfloor +1)$, since $\mathcal{A} \cap \mathcal{B} = \emptyset$. By Theorem \ref{thG1}, we have $\rank IC_{n}\leq \left\vert \mathcal{G}\right\vert
=4(\left\lfloor \frac{n}{4}\right\rfloor +1)$.
\end{proof}

For $k\in [n]$, let
\[IC_{n}(k)=\{\alpha \in IC_{n}\mid \rank\alpha \geq k\}.\]

\begin{remark}\label{re1}
Let $\alpha,\beta_{1},\beta_{2}\in IC_{n}$ with $\rank\alpha = k$ for some $k\in \{2,\ldots,n\}$. Then either the sizes as well as the parities of the
end points of the maximal intervals in $\dom\alpha$ and $\dom\beta_{1}\alpha\beta_{2}$ coincide or $\rank(\beta_{1}\alpha \beta_{2})<k$.
\end{remark}

\begin{notation}
Let $\alpha \in IC_{n}$ with $1,n \in \dom\alpha $, where $I_{(1)}$ and $I_{(n)}$ are the maximal intervals in $\dom\alpha$ with $1\in I_{(1)}$
and $n\in I_{(n)}$, respectively. For the rest of the paper, we consider $I_{(1)}\cup I_{(n)}$ as one maximal interval of $\dom\alpha$ instead of the
both maximal intervals $I_{(1)}$ and $I_{(n)}$.
\end{notation}

\begin{proposition}\label{prG3}
We have $\rank IC_{n}\geq 4(\left\lfloor \frac{n}{4}\right\rfloor +1)$.
\end{proposition}

\begin{proof}
Let $G^{\ast}$ be a generating set of $IC_{n}$. Since the permutation $\sigma_{1}\in IC_{n}$, the set $G^{\ast}$ contains a permutation $\alpha$. If $\alpha$ is orientation-preserving then the orientation-reversing permutation $\sigma_{2}$ belongs not to $\left\langle \alpha \right\rangle$ since the set of orientation-preserving transformations is closed. If $\alpha$ is not orientation-preserving then it is easy to verify that $\alpha$ has to be orientation-reversing and $\alpha^{2}=\id$.
Hence, $\sigma_{1}\notin \left\langle \alpha \right\rangle =\{\alpha,\id\}$. Therefore, $G^{\ast}$ contains at least two transformations with rank $n$.

Since $\varepsilon_{1}\in IC_{n}$ and the transformations of rank $n$ (permutations) forms a subgroup of $IC_{n}$, the set $G^{\ast}$ contains an $\alpha \in IC_{n}$ with rank $n-1$. Let $a\in [n]$ such that $[n] = \dom\alpha \cup \{a\}$. If $a$ is odd then $\varepsilon_{n}\notin \left\langle IC_{n}(n), \alpha \right\rangle$ and if $a$ is even
then $\varepsilon_{1}\notin \left\langle IC_{n}(n), \alpha \right\rangle$ because the transformations in $IC_{n}(n-1)$ preserve the parity (i.e. if $\beta \in IC_{n}(n-1)$ and $x\in \dom\beta$ then $x$ and $x\beta$ have the same parity). Therefore, $G^{\ast}$ contains at least two transformations with rank $n-1$.

Assume that there are less than $2\left\lfloor \frac{n}{4}\right\rfloor -2$ transformations with rank $n-2$ in $G^{\ast}$. It is easy to verify that
$\left\lfloor \frac{n}{4}\right\rfloor -1=\left\vert 2\mathbb{N}+1\cap \{3,4,\ldots ,\frac{n}{2}-1\}\right\vert$. Moreover, $\dom\alpha$, for some
$\alpha \in IC_{n}$ with rank $n-2$, cannot obtain two maximal intervals with a size in $2\mathbb{N}+1\cap \{3,4,\ldots ,\frac{n}{2}-1\}$. Hence,
there is an $a\in 2\mathbb{N}+1\cap \{3,4,\ldots ,\frac{n}{2}-1\}$ such that all $\alpha \in G^{\ast}$ with rank $n-2$ have no maximal interval of size $a$ in
$\dom\alpha$ with even endpoints or all $\alpha \in G^{\ast}$ with rank $n-2$ have no maximal interval of size $a$ in $\dom\alpha$ with odd endpoints.
Without loss of generality, we can assume that all $\alpha \in G^{\ast}$ with rank $n-2$ have no maximal interval of size $a$ in $\dom\alpha$ with odd endpoints.
Note that $\gamma_{n-(a+1),n}\in IC_{n}$ has a maximal interval $\{n-(a+1)+1,\ldots,n-1\}$ of size $a$ with odd end points. Hence, there are $\alpha_{1},\ldots,\alpha_{p}\in G^{\ast}$ such that $\alpha_{1}\cdots \alpha_{p}=\gamma_{n-(a+1),n}$. Assume that there is $i\in\{1,\ldots,p\}$ with $\rank\alpha_{i}=n-2$. Since $\dom\alpha_{i}$ has no maximal interval of size $a$ with odd end points, the set $\dom(\alpha_{1}\cdots \alpha_{p})=\dom\gamma_{n-(a+1),n}$ has no maximal interval of
size $a$ with odd end points by Remark \ref{re1}. This contradicts that $\{n-(a+1)+1,\ldots,n-1\}$ is a maximal interval of $\dom\gamma_{n-(a+1),n}$ with odd endpoints. This provides that $\alpha_{1},\ldots,\alpha_{p}\in IC_{n}(n-1)$. It is easy to verify that all transformations in $IC_{n}(n-1)$ are orientation-preserving or orientation-reversing. Since the
set of all orientation-preserving or orientation-reversing (partial) transformations is closed, we can conclude that $\gamma_{n-(a+1),n}$ is orientation-preserving or orientation-reversing, which gives a contradiction. Hence, there are at least $2\left\lfloor \frac{n}{4}\right\rfloor -2$ transformations with rank $n-2$ in $G^{\ast}$.

Assume that there are less than $2\left\lfloor \frac{n}{4}\right\rfloor$ transformations with rank $n-3$ in $G^{\ast}$. It is easy to verify that
$2\left\lfloor \frac{n}{4}\right\rfloor =\left\vert 2\mathbb{N}\cap \{\frac{n}{2},\frac{n}{2}+1,\ldots,n-4\}\right\vert$ and $\dom\alpha$, for any
$\alpha \in IC_{n}^{\overline{p}}$ with rank $n-3$, cannot obtain two maximal intervals with a size in $2\mathbb{N}\cap \{\frac{n}{2},\frac{n}{2}+1,\ldots,n-4\}$.
Thus, there is $r\in 2\mathbb{N}\cap \{\frac{n}{2},\frac{n}{2}+1,\ldots,n-4\}$ such that for all $\alpha \in G^{\ast}$ with rank $n-3$ and
$a\alpha \in 2\mathbb{N}+1\cap [n]$ for some $a\in 2\mathbb{N}\cap [n]$ or for all $\alpha \in G^{\ast}$ with rank $n-3$ and $a\alpha \in 2\mathbb{N}\cap [n]$ for some
$a\in 2\mathbb{N}+1\cap [n]$, the maximal intervals in $\dom\alpha$ have a size different from $r$. Without loss of generality, we can assume that for all $\alpha \in G^{\ast}$ with rank $n-3$ and $a\alpha \in 2\mathbb{N}+1\cap [n]$ for some $a\in 2\mathbb{N}\cap [n]$, the maximal intervals in $\dom\alpha$ have a size different from $r$. It is easy to verify that $2\delta_{\frac{n-r}{2}-1}^{o}=1$ and $\dom\delta_{\frac{n-r}{2}-1}^{o}$ has a maximal interval $\{4,5,\ldots,r+3\}$ of size $r$. Then there are
$\alpha_{1},\ldots,\alpha_{p}\in G^{\ast}$ such that $\alpha_{1}\cdots \alpha_{p}=\delta_{\frac{n-r}{2}-1}^{o}$ and there is $i\in \{1,\ldots,p\}$ such that
$a\alpha_{i}\in 2\mathbb{N}+1$ for some $a\in 2\mathbb{N}\cap [n]$. Then by Remark \ref{re1}, both $\dom\alpha_{i}$ and $\dom(\alpha_{1}\cdots \alpha_{p})$ have the same
maximal intervals, i.e. $\dom\delta_{\frac{n-r}{2}-1}^{o}=\dom(\alpha_{1}\cdots \alpha_{p})$ has no maximal interval of size $r$, a contradiction. Hence, there are at least $2\left\lfloor \frac{n}{4}\right\rfloor$ transformations with rank $n-3$ in $G^{\ast}$.

Let $U=\{\alpha \in IC_{n} \mid \dom\alpha =\{2,4,\ldots,n\}$ and $\im\alpha =\{1,3,\ldots,n-1\}\}$. Note that $U \subset IC_{n}$. In particular, we have $\eta_{2}\in U$. Then there are $\alpha_{1},\ldots,\alpha_{p}\in G^{\ast}$ and $\alpha_{0}=\id$ such that $\alpha_{0}\alpha_{1}\cdots \alpha_{p}=\eta_{2}$. Then there is $r\in \{0,\ldots,p-1\}$
such that $w_{1} = \alpha_{0}\cdots \alpha_{r}\notin U$ and $\alpha_{0}\cdots \alpha_{r}\alpha_{r+1}\in U$. We put $w_{2}=\alpha_{r+1}\in G^{\ast}$ and have
$\{2,4,\ldots,n\}=\dom(w_{1}w_{2})\subseteq \dom w_{1}$. Since the elements in $\{2,4,\ldots,n\}$ are pairwise not connected, the elements in $\{2,4,\ldots,n\}w_{1}$ are pairwise not connected. Hence $\{2,4,\ldots,n\}w_{1}=\{1,3,\ldots,n-1\}$ or $\{2,4,\ldots,n\}w_{1}=\{2,4,\ldots,n\}$. If $\{2,4,\ldots,n\}w_{1}=\{1,3,\ldots,n-1\}$ then we obtain $w_{1}\in U$ by Lemma \ref{leG10}. Since $w_{1}\notin U$, we have $\{2,4,\ldots,n\}w_{1}=\{2,4,\ldots,n\}$. This implies  $\{2,4,\ldots,n\}w_{2}=\{2,4,\ldots,n\}w_{1}w_{2}=\{1,3,\ldots,n-1\}$. Again by Lemma \ref{leG10}, we obtain $w_{2}\in U$. By the same argumentation, we can show that
there is $w\in G^{\ast}$ with $\dom w=\{1,3,\ldots,n-1\}$ and $\im w=\{2,4,\ldots,n\}$. Thus, $G^{\ast}$ contains at least two transformations of rank $\frac{n}{2}$.

Altogether, we have shown that
$$\rank IC_{n}\geq 2+2+2\left\lfloor \frac{n}{4}\right\rfloor-2+2\left\lfloor \frac{n}{4}\right\rfloor+2 =
4\left\lfloor\frac{n}{4}\right\rfloor+4 = 4\left(\left\lfloor\frac{n}{4}\right\rfloor +1\right).$$
\end{proof}

From Propositions \ref{prG4} and \ref{prG3}, we have
\begin{theorem}\label{thG2}
$\rank IC_{n}= 4(\left\lfloor \frac{n}{4}\right\rfloor +1)$.
\end{theorem}

\end{document}